\documentclass[12pt]{amsart}

\usepackage[margin=1in]{geometry}
\usepackage{amsmath,amssymb,amsthm,mathtools, bbm, color}
\usepackage{enumitem}
\usepackage{microtype}
\usepackage[hidelinks]{hyperref}

\newtheorem{theorem}{Theorem}[section]
\newtheorem{proposition}[theorem]{Proposition}
\newtheorem{lemma}[theorem]{Lemma}
\theoremstyle{definition}
\newtheorem{definition}[theorem]{Definition}
\newtheorem{remark}[theorem]{Remark}
\newtheorem{question}[theorem]{Question}
\newtheorem{corollary}[theorem]{Corollary}
\newtheorem{example}[theorem]{Example}

\newcommand{\R}{\mathbb R}
\newcommand{\Sph}{\mathbb S}
\newcommand{\Sym}{\operatorname{Sym}}
\newcommand{\tr}{\operatorname{tr}}
\newcommand{\Sing}{\operatorname{Sing}}

\newcommand{\cH}{\mathcal H}
\newcommand{\cP}{\mathcal P}

\newcommand{\norm}[1]{\lVert #1\rVert}
\newcommand{\abs}[1]{\lvert #1\rvert}

\newcommand{\epsi}{\varepsilon}

\title{On Singular Sets of Fully Nonlinear Uniformly Elliptic Equations}

\author[B.Z. Chu]{BaoZhi Chu}
\address[B.Z. Chu]{Department of Mathematics, Northwestern University,
2033 Sheridan Road, Evanston, IL 60208}
\email{bzchu@northwestern.edu}

\begin{document}

\begin{abstract}
For continuous viscosity solutions $u$ of fully nonlinear uniformly elliptic equations $F(D^2 u)=0$,
the work of Nadirashvili--Tkachev--Vl\u{a}du\c{t} shows that $u$ need not be $C^2$ in dimensions $n\ge 5$. 
It is therefore natural to ask how large the set $\Sing(u)$, consisting of points at which $u$ has no $C^2$ neighborhood, can be.
Under an additional $C^1$ assumption on $F$, Armstrong--Silvestre--Smart proved that $\Sing(u)$ has Hausdorff dimension at most $n-\varepsilon$ for some $\varepsilon >0$.

In this paper, we show that, in dimensions $n\ge 5$, the Hausdorff dimension of $\Sing(u)$ cannot be bounded away from $n$ under uniform
ellipticity alone. In fact, we prove the stronger statement: $\Sing(u)$ can be any compact nowhere dense set in $\R^5$ modulo a countable set.  Consequently, in every dimension $n\geq 5$, the Hausdorff dimension of $\Sing(u)$ can be any number in $[n-5, n]$; moreover, $\Sing(u)$ can even have positive Lebesgue measure.

In contrast, for every dimension $n$ and every uniformly elliptic
operator $F$, we prove that the set $\Sigma_{2,0}(u)$ of points at which $u$ is not twice differentiable has Hausdorff dimension at most $n-\varepsilon$ for some $\varepsilon>0$. In particular, this quantitatively strengthens Trudinger's theorem that $u$
is twice differentiable almost everywhere. 
More generally, we prove for every $\alpha\in [0,1)$ that  the set $\Sigma_{2,\alpha}(u)$ of points at which $u$ has no $C^{2,\alpha}$ expansion has Hausdorff dimension at most $n-\varepsilon(1-\alpha)$. 
\end{abstract}

\maketitle

\section{Introduction}

In dimension $n\geq 2$, we denote by $\Sym(n)$ the space of real
symmetric $n\times n$ matrices. 
Let $F:\Sym(n)\to \R$ be a \emph{uniformly elliptic} operator. Namely, there exist $\Lambda\ge \lambda>0$ such that
\begin{equation}\label{eq:UE}
\lambda\tr N\leq F(M+N)-F(M)\leq\Lambda\tr N\quad
\text{for every } M\in\Sym(n)\text{ and every }N\geq 0. 
\end{equation}

Let $u$ be a continuous viscosity solution of 
\begin{equation}\label{equ-u}
    F(D^2 u)=0\quad \text{in }B_1,
\end{equation}
where $B_1\subset\R^n$ is the open unit ball centered at the origin. 

A fundamental question in the study of elliptic equations is whether $u$ must be $C^2$. This question can be reformulated as follows. Define 
$$
 \Sing(u):=
 \left\{x\in B_1 \mid \text{there is no neighborhood $U\subset B_1$ of $x$
 such that $u\in C^2(U)$}\right\}.
$$
It follows that $\Sing(u)$ is relatively closed in $B_1$, and $u$ is $C^2$ in $B_1\setminus \Sing(u)$.
\begin{question}\label{Q-1}
    Does every continuous viscosity solution $u$ of \eqref{equ-u} satisfy $\Sing(u)=\varnothing$?
\end{question}
The answer in dimension two is positive due to a classical result of Nirenberg \cite{Nirenberg} in the 1950s. Over the past two decades, after a series of works progressively bringing down the dimension,
Nadirashvili--Tkachev--Vl\u{a}du\c{t} \cite{NTV} gave a negative answer to this question in dimensions five and higher. 
Question \ref{Q-1} remains open in dimensions three and four. An important point in this question is that we do not assume convexity or concavity on $F$, since otherwise $u$ must be $C^{2,\alpha}_{\mathrm{loc}}$ by the Evans--Krylov theorem \cite{Evans-82,Krylov-82}.
For background and related surveys on Question \ref{Q-1}, see, for example, \cite{Caffarelli1995FullyNE, XX-book, Mooney-review, NTV-book, Teixeira-survey} and references therein.

The solution $u$ in \cite{NTV} is given by
\begin{equation}\label{def-w}
 w(x):=
 P(x)/\abs{x},
\end{equation}
where $w(0)=0$, and  $P$ is the five-dimensional \emph{Cartan isoparametric cubic form} 
\begin{equation}\label{cartan-cubic}
 P(x):=x_1^3
 +\frac{3x_1}{2}
 \bigl(z_1^2+z_2^2-2z_3^2-2x_2^2\bigr)
 +\frac{3\sqrt3}{2}
 \bigl(x_2z_1^2-x_2z_2^2+2z_1z_2z_3\bigr).
\end{equation}
Here $x=(x_1,x_2,z_1,z_2,z_3)\in\R^5$. Clearly, $\Sing(w)=\{0\}$. 
To the best of our knowledge, in all previously known five-dimensional
examples, $\Sing(u)$ is either empty or a singleton.

Since we know that in general $\Sing(u)$ may not be empty, it is natural to ask
\begin{question}
    For a continuous viscosity solution $u$ of \eqref{equ-u}, how large can $\Sing(u)$ be? 
\end{question}
An answer to this question is given by Armstrong--Silvestre--Smart under the following additional assumption:
\begin{equation}\label{F-C1}
F\in C^1\text{ and its derivative }DF \text{ is uniformly continuous on }\Sym(n).
\end{equation}
Among other results, they prove 
\begin{theorem}[\cite{ASC}]
\label{thm-asc}
Let $F: \Sym(n)\to \R$ satisfy \eqref{eq:UE} and \eqref{F-C1}. Suppose that $u$ is a continuous viscosity solution of \eqref{equ-u}.  
Then $\dim_{\mathcal H} \Sing(u)\le n-\varepsilon$ for some $\varepsilon>0$ depending only on $n$, $\lambda$, and $\Lambda$. 
\end{theorem}
 Here $\dim_{\mathcal H}E$ denotes the standard Hausdorff dimension of the set $E$.
 Recall that the uniform ellipticity of $F$ alone implies that $F$ is Lipschitz continuous, 
but it does not imply the $C^1$ regularity required in \eqref{F-C1}. For example, Isaacs operators are generally not $C^1$ and
need not be either convex or concave.

\subsection{Compact nowhere dense singular sets}
It is an important open problem whether the Hausdorff dimension
estimate in Theorem \ref{thm-asc} remains valid without the assumption \eqref{F-C1}; see \cite[\S 4.5]{XX-book} for a survey.  
In the first part of this paper, we show that, for general uniformly elliptic $F$ without \eqref{F-C1}, the singular set can be large, and in particular, we give a negative answer to this open problem.
In fact, we prove a much stronger result: for every nonempty compact nowhere dense set $K\subset\R^5$, one can have $\Sing(u)=K\cup A$, where $A$ is countable. Our first main result is as follows.
\begin{theorem}\label{main-thm}
Let $K\subset B_1$ be any nonempty compact nowhere dense set in $\R^5$.  Then there
exist a countable set $A\subset B_1\setminus K$, a function
$u\in C^{1,1}(B_1)$, and a uniformly elliptic operator
$F:\Sym(5)\to\R$ such that $\Sing(u)=K\cup A$, and \eqref{equ-u} holds in the viscosity sense. 
Moreover, the following hold.
\begin{enumerate}[label=\textup{(\roman*)}]

\item The set of accumulation points of $A$ in $\R^5$ is exactly $K$, and every point
of $A$ is isolated in $\Sing(u)$.  

\item The function $u$ is smooth in $B_1\setminus \Sing(u)$ and is not twice
differentiable at any point of $A$.  At every $x\in K$, $u$ is
twice differentiable with
$D^2 u(x)=D^2 P(x)$, and 
$$ 
 u(x+h)
 =
 P(x)+DP(x)\cdot h+ 2^{-1} h^T D^2P(x) h
 +O(\abs{h}^3)
 \quad\text{as }h\to 0,
$$
where the constant in $O(\abs{h}^3)$ is independent of $x\in K$. 

\item The ellipticity constants of $F$ are
independent of $K$.
\end{enumerate}

\end{theorem}

By choosing the set $K$ in Theorem \ref{main-thm}, we immediately obtain the following:

\begin{corollary}
    For every real number $0 \le s\le 5$, there exist a function $u\in C^{1,1}(B_1)$ and a uniformly elliptic operator $F:\Sym(5)\to \R$ such that \eqref{equ-u} holds in the viscosity sense, and
    $$
 \dim_{\mathcal H}\Sing(u)=s.
$$
Moreover, when $s>0$, the set $\Sing(u)$ may be chosen to have either
zero or positive finite $s$-dimensional Hausdorff measure.
\end{corollary}

By adding dummy variables, this corollary shows that in every dimension $n\ge 5$, the Hausdorff dimension of $\Sing(u)$ can be any number in $[n-5,n]$. In particular,
the conclusion of Theorem \ref{thm-asc} fails for general uniformly elliptic $F$ without the assumption \eqref{F-C1} in dimensions five and higher. 

\medskip

Let us briefly describe our construction in Theorem \ref{main-thm}. We first choose a countable set
$A\subset B_1\setminus K$ whose set of accumulation points is exactly
$K$. 
For every $a\in A$,  
we choose pairwise disjoint balls $\overline{B(a,R_a)}\subset B_1\setminus K$ whose
radii decay sufficiently rapidly as their centers approach $K$.  Let
$V(x):=\chi(|x|) w(x)$, where $\chi$ is a fixed cutoff, and define
$$ \textstyle
 u=P+ \eta \sum_{a\in A}  R_a^3 
 V\bigl( R_a^{-1} (\cdot -a ) \bigr),
$$
where $\eta>0$ is a sufficiently small universal constant.
Each localized profile makes $u$ non-twice-differentiable at its center $a\in A$. On the other hand, the scaling and the geometry of the support balls ensure that the localized
sum has zero Hessian at every point of $K$.  Since the set of accumulation points of $A$ is exactly $K$, every point of $K$ nevertheless belongs to $\Sing(u)$. Consequently, $\Sing(u)=K\cup A$. 

Let $$\Omega:= B_1\setminus\Sing(u)\quad  \text{and}
\quad \mathcal H:=\overline{D^2u(\Omega)}.$$ 

Here and throughout the paper, for $M\in \Sym(n)$, we use its Frobenius norm
$\norm{M}_F:=(\tr(M^2))^{1/2}$, and order its eigenvalues as
$\lambda_1(M)\geq\cdots\geq\lambda_n(M)$.

One key step of our construction is to establish the following estimate: There is a universal constant $\theta>0$ such that
$$
 \lambda_1(X-Y)\geq\theta\norm{X-Y}_F
 \quad\text{for all }X,Y\in\mathcal H.
$$
In our estimate for the Hessian image, the orbit reduction property of $P$, which is used as in \cite{NTV}, reduces the analysis of the augmented Hessians 
$D^2 w(e)+ sD^2 P(e)$ to a lower-dimensional problem.

In contrast to the abstract operator construction in \cite{NTV, NV-gafa},
we define $F$ explicitly on $\Sym(5)$.
The preceding estimate, together with the compactness of $\mathcal{H}$, allows us
to define
$$ \textstyle
 F(M):=\min_{X\in\mathcal H}
 \bigl\{\tr(M-X)+\sqrt{5}\theta^{-1}\lambda_1(M-X)\bigr\},\quad M\in \Sym(5).
$$
The ellipticity constants of $F$ are universal, and in particular are independent of $K$. Since $F$ vanishes on $\mathcal H$, the equation $F(D^2 u)=0$ holds classically on $\Omega$.  We next show that
$D^2u(x)\in\mathcal H$ for every $x\in K$, even though $K\subset\Sing(u)$, and hence the equation holds
pointwise on $K$.  Finally, a removable singularity result for countable
sets extends the equation across $A$ in the viscosity sense.

In contrast
to \cite{NTV}, where a single homogeneous singular profile $w$ is used,
our construction localizes infinitely many copies of that profile at
different centers and scales. 
The background term $P$ in the definition of $u$ is indispensable in our construction.  Without
it, the pure sum of localized profiles has a smooth point at which its Hessian
vanishes and another smooth point at which its Hessian is nonzero and
negative semidefinite.  Uniform ellipticity prevents these two Hessians
from lying on the same zero level set of any uniformly elliptic operator $F$.

In fact, our proof shows that one may take $\theta=44^{-1}$ in the above construction. Thus the operator $F$ constructed above is uniformly elliptic with
 $\lambda=1$ and $\Lambda=1+44\sqrt{5}$. Note that if $\Lambda/\lambda$ is sufficiently close to $1$, then Cordes--Nirenberg type estimates imply that every continuous viscosity solution $u$ of \eqref{equ-u} must be $C^{2,\alpha}_{\mathrm{loc}}$ for some $\alpha\in (0,1)$.

\subsection{Hausdorff dimension bounds for pointwise singular sets}
In dimensions $n\ge 5$, we have shown that, without additional assumptions on $F$, the Hausdorff dimension of
$\Sing(u)$ cannot be bounded away from the ambient dimension $n$. 
Meanwhile, Theorem \ref{main-thm} exhibits a distinction between local $C^2$ regularity and
pointwise twice differentiability: every point of the prescribed compact nowhere dense
set $K$ belongs to $\Sing(u)$, but $u$ is still twice differentiable there.
This motivates the following pointwise singular sets.

\begin{definition}
\label{def:expansions}
Let $0\leq\alpha<1$, and let $x\in B_1$.  We say that a continuous
function $u$ in $B_1$ has a \emph{$C^{2,\alpha}$ expansion at $x$} if there
exist $p_x\in\R^n$ and $A_x\in\Sym(n)$ such that
$$
 u(x+h)=u(x)+p_x\cdot h+ 2^{-1} h^TA_xh+o(\abs{h}^{2+\alpha})
 \quad\text{as }h\to 0.
$$
We define
$$
 \Sigma_{2,\alpha}(u)
 :=\left\{x\in B_1 \mid \text{$u$ has no $C^{2,\alpha}$ expansion at $x$}\right\}.
$$
\end{definition}
It follows that
$\Sigma_{2,\alpha}(u)\subset\Sigma_{2,\beta}(u)$ whenever $0\leq\alpha<\beta<1$, and 
$\Sigma_{2,0}(u)\subset\Sing(u)$.
 
When $\alpha=0$, our notion of
$C^{2,0}$ expansion is exactly what is called
\emph{punctual second order differentiability} in
\cite[Definition~1.4]{Caffarelli1995FullyNE}.
For a $C^{1,1}$ function, the existence of a $C^{2,0}$ expansion at
a point implies twice differentiability at that point; see
Lemma~\ref{prop:punctual-versus-twice}.  For $C^{1,\alpha}$ functions
with $\alpha<1$, however, this implication fails in general; see
Example~\ref{ex-b2}.  On the other hand, for continuous viscosity
solutions of \eqref{equ-u}, these two notions are equivalent by
Lemma~\ref{lem:punctual-implies-second}.  Consequently, for every solution $u$ considered in this paper,
$$
 \Sigma_{2,0}(u)
 =
 \bigl\{x\in B_1\mid
 \text{$u$ is not twice differentiable at $x$}\bigr\}.
$$

The set $\Sigma_{2,\alpha}(u)$ is a $G_{\delta\sigma}$ set by Lemma \ref{Lem-G-set}. 
It need not, however, be
relatively closed in $B_1$, even when $u$ is a $C^{1,1}$ viscosity solution of \eqref{equ-u}. Indeed, let 
$u$ be a solution constructed in
Theorem \ref{main-thm}. Then, for every $\alpha\in [0,1)$, we have
$$\Sigma_{2,\alpha}(u)=A \quad \text{but}\quad 
\overline{\Sigma_{2,\alpha}(u)}
=A\cup K
 \neq
 \Sigma_{2,\alpha}(u).$$ Hence, $\Sigma_{2,\alpha}(u)$ is not relatively closed in $B_1$. 
 Moreover, noting that $\Sing(u)=A\cup K$, this example also indicates that $\Sigma_{2,\alpha}(u)$ can remain very small even when $\dim_\cH\Sing(u)=n$. 
Indeed, we establish a positive Hausdorff codimension bound for $\Sigma_{2,\alpha}(u)$ for each fixed $\alpha<1$. 
\begin{theorem}
\label{thm:main-2}
Let $F:\Sym(n)\to\R$ satisfy \eqref{eq:UE}, and let $u$ be a continuous
viscosity solution of \eqref{equ-u}.  Then there exists
$\varepsilon\in(0,1]$, depending only on $n$, $\lambda$, and $\Lambda$, such that
$$
 \dim_{\mathcal H}\Sigma_{2,\alpha}(u)
 \leq n-\varepsilon(1-\alpha)\quad \text{for every}\,\,\, 0\leq\alpha<1.
$$
\end{theorem}

\begin{remark}\label{remark-estimate}
In Theorem~\ref{thm:main-2}, the restriction $\alpha<1$ cannot be
removed if one seeks a nontrivial estimate for
$\dim_{\cH}\Sigma_{2,\alpha}(u)$, namely, an upper bound strictly
smaller than $n$.  Indeed, after extending
Definition~\ref{def:expansions} to $\alpha=1$, the set
$\Sigma_{2,1}(u)$ can be the whole ball $B_1$.  Such examples of $u$ include all nonzero homogeneous harmonic polynomials of degree three.
\end{remark}

Trudinger's theorem \cite{Trudinger-twice-differentiability} states that a continuous viscosity solution $u$ of \eqref{equ-u} must be almost everywhere twice differentiable; equivalently, the Lebesgue measure of $\Sigma_{2,0}(u)$ is zero. 
His result in fact applies to a very general class of equations.
It does not, however, provide a positive lower bound for the Hausdorff codimension of $\Sigma_{2,0}(u)$.  
Therefore,
Theorem~\ref{thm:main-2}, in the case $\alpha=0$, quantitatively strengthens Trudinger's theorem. 

Our proof of Theorem \ref{thm:main-2} needs the
$W^{3,\varepsilon}$ estimate of Armstrong--Silvestre--Smart.
Roughly speaking, this estimate comes from applying to $Du$ the
$W^{2,\varepsilon}$ estimate of Lin \cite{Lin}; see
Proposition \ref{prop:third-order-tail} and the discussion there.
One of our new points is that, under uniform ellipticity alone, this estimate
can be combined with quadratic approximation at decreasing dyadic scales to control the
set where pointwise $C^{2,\alpha}$ expansions fail.  A further
scale-by-scale covering argument then yields the Hausdorff codimension
$\varepsilon(1-\alpha)$ in Theorem~\ref{thm:main-2}.

It remains natural to ask whether the factor $1-\alpha$ can be removed under uniform ellipticity alone, or more generally to determine the optimal dependence of $\dim_\cH \Sigma_{2,\alpha}(u)$ on $\alpha$.
Under the additional assumption \eqref{F-C1},
Armstrong--Silvestre--Smart \cite[Theorem~1.1]{ASC} prove
a stronger conclusion: there exists a closed set $\Sigma$ of Hausdorff dimension at most $n-\varepsilon$ such that
$u\in C^{2,\beta}(B_1\setminus\Sigma)$
for every $0<\beta<1$. As a consequence, $\dim_{\mathcal H}\Sigma_{2,\alpha}(u)\leq n-\varepsilon$ for every $\alpha\in [0,1)$.
The distinction is that
Theorem~\ref{thm:main-2} assumes only uniform ellipticity and makes no
differentiability assumption on $F$.

\medskip

The paper is organized as follows. Section \ref{sec-cartan} establishes Proposition
\ref{augmented-family-proposition}, together with other properties of the Cartan cubic needed in the proof. Section \ref{sec-thm-1}
proves Theorem \ref{main-thm}. Section \ref{sec:proof-positive-result} proves Theorem \ref{thm:main-2}. 
Appendix \ref{appendix-orbit-reduction} provides a proof of
the orbit reduction lemma used in Section \ref{sec-cartan}.
Appendix \ref{app-expansion} discusses pointwise $C^{2,\alpha}$ expansion. Throughout the paper, both $B_r(x)$ and $B(x,r)$ denote the open ball centered at $x$ of radius $r$.

\section{Cartan Isoparametric Cubic Form and its augmented Hessians}\label{sec-cartan}
Recall the functions $w, P$ defined in \eqref{def-w} and \eqref{cartan-cubic}.  
We denote by $\Sph^4$ the four-dimensional unit sphere. 
For $e\in\Sph^4$ and $s\ge 0$, let
$$
 K(e,s):=D^2w(e)+sD^2P(e).
$$

One crucial new ingredient in our proof of Theorem \ref{main-thm} is the following 
\begin{proposition}
\label{augmented-family-proposition}
For $e,f\in\Sph^4$ and $s,t\geq 0$, 
$$
 \lambda_1\bigl(K(e,s)-K(f,t)\bigr)
 \geq  {44}^{-1}\norm{K(e,s)-K(f,t)}_F.
$$
\end{proposition}

The proof of Proposition \ref{augmented-family-proposition} uses the
following orbit reduction property of $P$. 
Let
$$
 e_p:=(p,0,\sqrt{1-p^2},0,0)\in\Sph^4\quad \text{for }p\in [-1,1].
$$

\begin{lemma}[\cite{NTV}]\label{orbit-reduction}
For every $e\in\Sph^4$, there is a unique $p\in[-1,1]$, and there exists an orthogonal
matrix $O\in O(5)$ such that
$e=Oe_p$ and $P\circ O =P$.
Moreover, $p$ is the unique solution in $[-1,1]$ of
$P(e)= p(3-p^2)/2$.
\end{lemma}

For completeness, a proof of Lemma~\ref{orbit-reduction} is included
in Appendix~\ref{appendix-orbit-reduction}.

\begin{proof}[Proof of Proposition~\ref{augmented-family-proposition}]
We first analyze the eigenvalues for general $K(e_p,s)$. 

For $p\in [-1, 1]$, 
let $\rho=\sqrt{1-p^2}$. A direct calculation gives $K(e_p,s)=\operatorname{diag}\{B_3, B_2\}$, where
$$
2B_3=
\begin{pmatrix}
 p(3+4p^2-3p^4+12s)&3\sqrt3p(p^2-1)&3\rho(1-p^4+2s)\\
 3\sqrt3p(p^2-1)&p(p^2-15-12s)&3\sqrt3\rho(1+p^2+2s)\\
 3\rho(1-p^4+2s)&3\sqrt3\rho(1+p^2+2s)&p^3+3p^5+6ps
\end{pmatrix}
$$
and
$$
2B_2=
\begin{pmatrix}
 p(p^2+3+6s)&6\sqrt3\rho(s+1)\\
 6\sqrt3\rho(s+1)&p(p^2-15-12s)
\end{pmatrix}.
$$

A computation shows that the eigenvalues of $B_2$ are 
$b_\pm$, given by
$$
 2b_\pm(p,s) =
 {p(p^2-3s-6)\pm3\sqrt3(s+1)\sqrt{4-p^2}} 
$$
and the eigenvalues of $B_3$ are $m$ and $a_\pm$, given by
 $$
 2 m(p,s):= {p(p^2+6s+3)},
 \quad \text{and} \quad 4  a_\pm(p,s)={5p^3-15p\pm3\sqrt{R(p,s)}},
$$
where 
$$
 R(p,s):=(8s+4)^2+p^2(3-p^2)^2(8s+5)>0.
$$
Here the computation for $B_3$ is more manageable by noting that $(-\sqrt3\rho,\rho,\sqrt3p)$ is an eigenvector associated with $m$.  
It follows from the above formulae that
$$
 2(b_+-m)
 = 3(s+1)
 \bigl(\sqrt{3(4-p^2)}-3p\bigr)\geq 0,
$$
$$
 2(m-b_-)
 = 3 (s+1)
 \bigl(\sqrt{3(4-p^2)}+3p\bigr)\geq 0.
$$
Moreover,
$$
 4(a_+-m)
 = 3 \bigl(\sqrt R-p(4s+7-p^2)\bigr)\geq 0,
$$
$$
 4(m-a_-)
 = 3 \bigl(\sqrt R+p(4s+7-p^2)\bigr)\geq 0.
$$
The only non-obvious nonnegativity assertion above follows from
$$
 R-p^2(4s+7-p^2)^2
 =
 4(4-p^2)
 \bigl(4s^2+4s+1-p^4
 +2p^2(1-p^2)s\bigr)\geq0.
$$

The above implies that, for every $p\in [-1,1]$ and $s\ge 0$, the eigenvalues of $K(e_p,s)$ satisfy
\begin{equation}\label{K-matrix-eigen-order}
    \max\{a_-, b_-\} \le    m\le \min\{a_+, b_+\}. 
\end{equation}
In particular,
$m=\lambda_3(K(e_p, s))$ is the middle eigenvalue.

For $p\in [-1,1]$ and $s\ge 0$, define
$$
 D(p,s):=(a_+-a_-)+(b_+-b_-)
 = {3} \sqrt{R/4}+3\sqrt3(s+1)\sqrt{4-p^2}
$$
and
$$
 T(p):=\tr K(e_p,s)=4p(p^2-3).
$$
Note that the trace of $K(e_p,s)$ is independent of $s$ since $D^2 P$ is traceless. Then \eqref{K-matrix-eigen-order} implies
\begin{equation}\label{extreme-pair-sums}
 (\lambda_1+\lambda_2)(K(e_p,s))=(T-m+D)/2 \quad \text{and}\quad (\lambda_4+\lambda_5)(K(e_p,s))=
 (T-m-D)/2.
\end{equation}
Moreover, $T$ is strictly decreasing on $[-1,1]$ and
$\abs{T'}\leq 12$ there.

Next we prove the following derivative bounds:
\begin{equation}\label{D-derivative-bounds}
 D_s(p,s)\geq20
 \quad\text{and}\quad
 D_p(p,s)<7,
 \qquad
 0\leq p\leq 1,\quad s\geq 0.
\end{equation}
In the proof of these bounds, we rewrite $R$ as
$$
 R=\sigma^2+h(\sigma+1),
$$
where
$\sigma:=8s+4$ and $h:=p^2(3-p^2)^2$.

Differentiating $D$ with respect to $s$ gives
$$
 D_s
 =
 6 R^{-1/2}(2\sigma+h)
 +3\sqrt3\sqrt{4-p^2}.
$$
Since 
$4R-(2\sigma+h)^2=h(4-h)\le 4$
and
$R\geq\sigma^2\geq16$, we obtain
$$
 R^{-1/2 }{(2\sigma+h)}
 \geq  {\sqrt{15}}/2.
$$
Combining the above, we obtain 
$$
 D_s
 \geq 3\sqrt{15}+9>20.
$$

Next we derive the bound for $D_p$. By differentiation, 
$$
 D_p
 =
 \frac{3}{4}\frac{R_p}{\sqrt R}
 -3\sqrt3(s+1)\frac{p}{\sqrt{4-p^2}}\le  \frac{3}{4}\frac{R_p}{\sqrt R},
$$
where we have used $p\ge 0$. A computation yields
$$
 R_p
 =
 6p(1-p^2)(3-p^2)(8s+5).
$$
Using
$\sqrt R\geq 8s+4$,
we obtain
 $$
 D_p
 \leq
 \frac{3^3}{2~}
 p(1-p^2)
 \frac{8s+5}{8s+4}
 \leq
 \frac{3^3}{2~}
 \cdot
 \frac{2}{3\sqrt3}\cdot \frac{5}{4}
 =
 \frac{15\sqrt3}{4}
 <7.
$$
Hence, we have proved both bounds in \eqref{D-derivative-bounds}.

Now we are ready to prove the desired estimate in Proposition \ref{augmented-family-proposition}.
Fix $e,f\in\Sph^4$ and $s,t\geq 0$. For the rest of the proof, set
$$
 A=K(e,s) \quad \text{and}
 \quad
 B=K(f,t). 
$$
\noindent\textbf{Claim.}
If
\begin{equation}\label{approximate-order}
 A\le  B +\gamma I_5,
\end{equation}
for some $\gamma\ge 0$, then
\begin{equation}\label{approximate-order-estimate}
 \norm{A-B}_F \leq  44 \gamma.
\end{equation}

Assuming the Claim at this moment, we prove the desired estimate. If $\lambda_1(A-B)\leq0$, then
$A\le B$. Applying the Claim with $\gamma=0$ gives $A-B=0$. The desired estimate clearly holds in this case. 
If $\gamma:=\lambda_1(A-B)>0$, then, by
$A-B\le \gamma I_5$, applying the Claim again gives
$\norm{A-B}_F\leq 44 \lambda_1(A-B)$. The desired estimate follows.

Now we prove the Claim. By Lemma \ref{orbit-reduction}, there exist $p,q\in [-1,1]$ and $O_f,O_e\in O(5)$ such that
$$
P\circ O_f=P,\quad P\circ O_e=P,\quad 
 f=O_f e_p,
 \quad \text{and}\quad 
 e=O_e e_q.
$$
Therefore, $w=P/\abs{x}$ is also invariant under
$O_f$ and $O_e$. From these facts, we obtain
$$
 B=O_f K(e_p,t) O_f^T\quad  \text{and}
 \quad
 A=O_e K(e_q,s) O_e^T.
$$
In particular, $\lambda_i(B)=\lambda_i(K(e_p, t))$ and $\lambda_i(A)=\lambda_i(K(e_q, s))$ for every $i$. 

Let
$$
 \Delta T:=T(p)-T(q),
 \quad
 \Delta m:=m(p,t)-m(q,s),
 \quad\text{and}\quad 
 \Delta D:=D(p,t)-D(q,s).
$$
Then, by \eqref{approximate-order}, an application of the min--max formulae for eigenvalues gives
$$
 \lambda_i(B)+\gamma\geq\lambda_i(A)
 \quad \text{for every }i.
$$
From \eqref{K-matrix-eigen-order}, \eqref{extreme-pair-sums}, and the above, we obtain
$$
 \Delta m\geq-\gamma,
\quad 
 {\Delta T-\Delta m+\Delta D}\geq - 4\gamma,
 \quad\text{and}\quad 
 {\Delta T-\Delta m-\Delta D} \geq - 4\gamma.
$$
It follows that
\begin{equation}\label{Delta-D-upper}
 \abs{\Delta D}\leq\Delta T+5\gamma.
\end{equation}

Now we prove
\begin{equation}\label{Delta-T-upper-bound}
     \Delta T\leq 36 \gamma.
\end{equation}
When $p\geq q$, the monotonicity of $T$ gives $\Delta T\le 0$; then \eqref{Delta-T-upper-bound} is clear. Hence, we may assume 
$$
 \delta:=q-p>0.
$$
To prove \eqref{Delta-T-upper-bound}, we only need to prove
$$\delta\leq3\gamma,$$  
since then the desired bound of $\Delta T$ follows from
$\abs{T'}\leq 12$.

If $p<0\leq q$, then by definition of $m$,
$$
 m(p,t)\leq 3p/2\quad \text{and}\quad 
 m(q,s)\geq 3q/2.
 $$
Thus $\Delta m\leq-3\delta/2$, and therefore
$\delta\leq2\gamma/3 \le 3\gamma$.

Suppose that $p<q\leq 0$.  Since both $P$ and $w$ are odd,
$K(-g,r)=-K(g,r)$ for every $g\in\Sph^4$ and $r\geq 0$.  Replacing the
ordered pair $(B,A)$ by
$$
 (B',A'):=(-A,-B)=(K(-e,s),K(-f,t))
$$
preserves both the inequality $B+\gamma I_5\geq A$ and the norm $\norm{A-B}_F$. By the ``Moreover'' part of Lemma \ref{orbit-reduction} and the oddness of $P$, the new orbit parameters are $p'=-q$ and $q'=-p$. Since $q'> p'\ge 0$, after this
relabeling argument, it remains only to consider the last case
$$
 q> p\ge 0.
$$

A direct computation gives
\begin{equation}\label{Delta-T-formula}
 \Delta T=4\delta(3-S),
\end{equation}
where $S:=p^2+pq+q^2$. 
Moreover, $\Delta m\geq-\gamma$ is equivalent to
\begin{equation}\label{middle-approximate-order}
 6p(t-s)\geq\delta(S+3+6s)-2\gamma.
\end{equation}
If $t\leq s$, this immediately gives
$\delta\leq2\gamma/3\le 3\gamma$.  So we may assume $t>s$.  By
\eqref{D-derivative-bounds},
$$
 \Delta D
 =D(p,t)-D(p,s)+D(p,s)-D(q,s)
 \geq20(t-s)-7\delta.
$$
Together with \eqref{Delta-D-upper}, this yields
$$
 20(t-s)\leq7\delta+\Delta T+5\gamma.
$$
Combining this inequality with \eqref{middle-approximate-order}, and
then using \eqref{Delta-T-formula}, gives
$$
 \left(\left(1+\frac{6p}{5}\right)S
 +3+6s-\frac{57p}{10}\right)\delta
 \leq\left(2+ 3p/2 \right)\gamma.
$$
Since $S\geq 3p^2$, the coefficient on the left hand side is bounded from below by
$$
 3p^2\left(1+\frac{6p}{5}\right)
 +3-\frac{57p}{10}
 =\frac{27}{20}
 +\frac{3}{20}(2p-1)^2(6p+11)
 \geq\frac{27}{20}.
 $$
Since $2+3p/2\leq7/2$, we conclude that
$\delta\leq {70} \gamma/27\leq 3\gamma$.
Hence, \eqref{Delta-T-upper-bound} is proved.

Finally, by the triangle inequality and \eqref{Delta-T-upper-bound},
 \begin{align*}
 \norm{A-B}_F
 &\leq\norm{B-A+\gamma I_5}_F+\sqrt5\gamma\\
 &\leq\tr(B-A+\gamma I_5)+\sqrt5\gamma=\Delta T+(5+\sqrt5)\gamma \leq 44\gamma.
 \end{align*}
The second inequality above uses 
$B-A+\gamma I_5\geq 0$. 
This proves \eqref{approximate-order-estimate}, and hence the Claim. This completes the proof of Proposition \ref{augmented-family-proposition}. 
\end{proof}

We end this section by proving 
\begin{equation}\label{background-Hessian-estimate}
 \lambda_1(D^2P(x))
 \geq c_0\abs{x},\quad x\in \R^5,
\end{equation}
where $c_0:=4^{-1} \sqrt{126/5}$. This estimate is needed in our proof of Theorem \ref{main-thm}.

To see this, it is well known that $P$ satisfies the Cartan--M\"unzner system
$$
 \abs{\nabla P(x)}^2=9\abs{x}^4,\qquad \Delta P=0.
$$
In particular, $\tr D^2P\equiv 0$. It follows that $4\lambda_1(D^2 P)\ge |\lambda_i(D^2 P)|$ for every $i$. From this and the fact $$\norm{D^2P(x)}_F^2=126\abs{x}^2,$$ 
we obtain \eqref{background-Hessian-estimate}.

\section{Proof of Theorem~\ref{main-thm}}\label{sec-thm-1}

Let $K\subset B_1$ be a nonempty compact nowhere dense set in $\R^5$. 

Our proof is divided into four steps.
Step 1 constructs the countable set $A$, and proves that its set of accumulation points is exactly $K$, and that every point of
$A$ is isolated in $E\coloneqq K\cup A$. In Step 2, we construct the function $u\in C^{1,1}(B_1)$, and prove assertion (ii) and thus $\Sing(u)=E$.  
 Step 3 constructs a uniformly elliptic
operator $F$ whose ellipticity constants are independent of $K$, and proves $F(D^2 u)=0$ on $B_1\setminus E$ classically. 
Finally, 
Step 4
proves $F(D^2 u)=0$ in the whole ball in the viscosity sense. 

\noindent\textbf{Step 1. Construction of the countable set $A$.}
  For every
$m\geq1$, let
$$
 \delta_m:=2^{-m-4}\operatorname{dist}(K,\partial B_1) >0.
$$
Since $K$ is compact, there are finitely many points
$\xi_{m,1},\ldots,\xi_{m,N_m}\in K$ such that
$K\subset\cup_{k=1}^{N_m}B_{\delta_m}(\xi_{m,k})$.
Since $K$ is closed and has empty interior, each set
$B_{\delta_m}(\xi_{m,k})\setminus K$ is open and contains uncountably many points.  We may therefore choose, for every $m$ and $k$,
$a_{m,k}\in B_{\delta_m}(\xi_{m,k})\setminus K \subset B_1$
 such that all the points $a_{m,k}$ are distinct.

Define
$$
 A:=\cup_{m\geq1}A_m \quad \text{with}\quad
 A_m:=\{a_{m,1},\ldots,a_{m,N_m}\}.
$$
Then the sets $A_m$ are finite and pairwise disjoint. In particular, $A\subset B_1\setminus K$ is countable. Moreover, we observe that for every $m$,
\begin{equation}\label{center-estimates}
 0<\operatorname{dist}(a,K)<\delta_m\text{ for }a\in A_m,
 \quad\text{and }\quad 
 \operatorname{dist}(x,A_m)<2\delta_m \text{ for }x\in K.
\end{equation}
Indeed, the first estimate follows from the choice of $a_{m,k}$.  Fix $x\in K$ and choose $k$ such that $x\in B_{\delta_m}(\xi_{m,k})$. Then
$\abs{x-a_{m,k}}
 \leq\abs{x-\xi_{m,k}}+\abs{\xi_{m,k}-a_{m,k}}
 <2\delta_m$, and the second estimate follows.

We now determine all accumulation points of $A$.  
Let $a_j\to z\in\R^5$ be a sequence of pairwise distinct points of $A$. Since each $A_m$ is finite, after passing to a subsequence, we may assume $a_j\in A_{m_j}$ with $m_j\to\infty$.  The first estimate in
\eqref{center-estimates} gives
$\operatorname{dist}(a_j,K)\to0$, and hence $z\in K$ by the closedness of $K$.  Conversely, let $x\in K$.  For every $m$, the second estimate in
\eqref{center-estimates} gives $a_m\in A_m$ such that
$\abs{a_m-x}<2\delta_m$, and thus $a_m \to x$ as $m\to\infty$. Hence, $x$ is an accumulation point of $A$. We have therefore proved that the set of accumulation points of $A$ is exactly
$K$.

We finish Step 1 by proving that every point of $A$ is isolated in the closed set
$$ E\coloneqq K\cup A. $$
Namely, fix $a\in A$. We prove that $\operatorname{dist}(a, E\setminus\{a\})>0$.
Let $d:=\operatorname{dist}(a,K)>0$.  Choose large $M$ such that
$\delta_m<d/3$ for every $m\geq M$. From this and \eqref{center-estimates}, we obtain 
$$
 \abs{a-b}\geq\operatorname{dist}(a,K)-\operatorname{dist}(b,K)
 >{2d}/{3}\quad \text{for all }b\in A_m,~m\ge M.
$$
Since the set $\cup_{m<M}A_m\setminus \{a\}$ is finite (possibly empty), the distance from $a$ to this set must be strictly positive. Combining the above facts, we obtain $\operatorname{dist}(a, E\setminus\{a\})>0$.  

\medskip
\noindent\textbf{Step 2. Construction of the solution $u$ and analysis of its Hessians.}
Let
$$
 R_a:={10}^{-1}\min\bigl\{
 \operatorname{dist}(a,E\setminus\{a\}),
 \operatorname{dist}(a,\partial B_1)\bigr\}>0 \quad \text{for }a\in A.
$$
Then, by definition and \eqref{center-estimates}, 
\begin{equation}\label{radius-bounds}
     R_a\leq {10}^{-1}\operatorname{dist}(a,K)< {10}^{-1}\delta_m  \quad \text{if }a\in A_m.
\end{equation}
Moreover,
if $a,b\in A$ and $a\neq b$, then
$R_a,R_b\leq\abs{a-b}/10$.  As a consequence,
\begin{equation}\label{radius-bounds-2}
 R_a+R_b\leq 5^{-1}\abs{a-b},\quad a,b\in A,~a\neq b. 
 \qquad
\end{equation}
It follows from the above that the closed balls $\overline{B(a,R_a)}\subset B_1\setminus K$ are pairwise disjoint.

Choose $\chi\in C^\infty[0,\infty)$ such that $\chi=1$ on $[0,1/4]$
and $\chi=0$ on $[1,\infty)$, and define
$$
 V(x):=\chi(\abs{x})w(x).
$$
Since $w\in C^{1,1}(\R^5)$, we have
$V\in C^{1,1}(\R^5)$.  The function $V$ is supported on
$\overline{B_1}$, equals $w$ in $B_{1/4}$, and is smooth away from the
origin. 
Therefore, there exists a constant
$C>0$, depending only on $P$ and $\chi$, such that
$$
|V|+|D V|+
\norm{D^2V}_F\leq C\quad \text{in } \overline{B_1}\setminus\{0\}.$$
Moreover, there exists a constant
$L>0$, depending only on $P$ and $\chi$, such that
$$
 \norm{D^2V(x)-D^2V(y)}_F \leq L \abs{x-y}
$$
whenever $x,y\neq0$ and at least one of $x,y$ lies outside $B_{1/4}$.
To see this, first suppose that $\abs{x-y}\geq 1/16$, then the estimate
follows from the boundedness of $D^2V$. If
$\abs{x-y}<1/16$, since one of $x$ and $y$ has norm at least $1/4$, then
every point on the segment joining $x$ and $y$ has norm at least $3/16$. The desired estimate follows from a bound of the third derivatives of $V$ on $\{|z|\ge 3/16\}$.

Choose, once and for all, a constant $\eta$ satisfying
$$
 0<4 \eta\leq {c_0}/{L},
$$
where $c_0$ is the constant in
\eqref{background-Hessian-estimate}.

Define
$$
u\coloneqq P+W  \quad \text{and}\quad 
\textstyle 
W:=\sum_{a\in A}V_a\quad \text{in }B_1,
$$
where
$$
 V_a(x):=\varepsilon_a R_a^2
 V\bigl( {(x-a)/R_a} \bigr)
 \quad\text{and}\quad 
 \varepsilon_a:=\eta R_a>0 \quad \text{for }a\in A. 
$$
The function $V_a$ is supported on $\overline{B(a,R_a)}$.  Since these
closed balls are pairwise disjoint, the sum in $W$ is well-defined and
at every point at most one summand is nonzero.  

For $a\in A$, scaling the estimates of $V$ gives
\begin{equation}\label{scaled-bump-estimates}
 \norm{V_a}_{L^\infty}\leq C\eta R_a^3,
 \quad
 \norm{DV_a}_{L^\infty}\leq C\eta R_a^2,
 \quad  
 \operatorname{Lip}(DV_a)\leq C\varepsilon_a, \quad \text{and}
\end{equation}
\begin{equation}\label{core-and-cross-estimates-2}
 \begin{gathered}
 \norm{D^2V_a(x)-D^2V_a(y)}_F
 \leq L\eta\abs{x-y}\\
 \text{if $x,y\neq a$ and at least one of them lies outside
 $B(a,R_a/4)$.}
 \end{gathered}
\end{equation}

Moreover, for $a\in A$, the homogeneity of $w$ implies
\begin{equation}\label{core-and-cross-estimates-1}
 V_a(x)=\varepsilon_aw(x-a)
 \quad\text{if }\abs{x-a}<R_a/4.
\end{equation}

\medskip
\noindent\textbf{Claim 1.}
The function $u$ belongs to $C^{1,1}(B_1)$ and is smooth in
$B_1\setminus E$.  At every $a\in A$, it is not twice differentiable.
At every $x\in K$, the function $u$ is twice differentiable with
$D^2u(x)=D^2P(x)$, and
\begin{equation}\label{cubic-expansion}
 \left|
 u(x+h)-u(x)-DP(x)\cdot h-2^{-1} h^TD^2P(x)h
 \right|
 \leq C\abs{h}^3,
\end{equation}
holds for all
$\abs{h}<\operatorname{dist}(x,\partial B_1)$, where $C$ is
independent of $x\in K$.
Consequently, $\Sing(u)=E$.

\begin{proof}[Proof of Claim 1]
We first observe that every point of
$B_1\setminus K$ has a neighborhood intersecting at most one of the closed support balls $\overline{B(a,R_a)}$.

We prove this observation by contradiction. If not, then
for some $x\in B_1\setminus K$, every neighborhood of $x$ intersects at least two distinct closed support balls.
Now we inductively choose pairwise distinct points $a_j\in  A$ and points
$y_j\in\overline{B(a_j,R_{a_j})}$
such that
$$
 \abs{y_j-x}< 1/j.
$$
Suppose that $a_1,\dots,a_{j-1}$ have already been chosen.  Since the
closed support balls are pairwise disjoint, $x$ belongs to 
$\overline{B(a_i,R_{a_i})}$ for at most one
of $1\le i\le j-1$.
Then we may choose small $\rho_j \in (0,1/j)$ such that
$
 B_{\rho_j}(x)\cap\overline{B(a_i,R_{a_i})}
 =\varnothing
$
for every $1\leq i\leq j-1$ such that
$x\notin\overline{B(a_i,R_{a_i})}$. 
Thus $B_{\rho_j}(x)$ intersects at most one of the previously chosen
support balls.  On the other hand, by the contradiction assumption, $B_{\rho_j}(x)$ intersects some support ball
$\overline{B(a_j,R_{a_j})}$ with
$a_j\notin\{a_1,\dots,a_{j-1}\}$.
Choose
$y_j\in
 B_{\rho_j}(x)\cap\overline{B(a_j,R_{a_j})}$.
This completes the induction.

Write $a_j\in A_{m_j}$.  Since every $A_m$ is finite and the points
$a_j$ are pairwise distinct, after passing to a subsequence, we have $m_j\to \infty$.
By \eqref{radius-bounds},
$R_{a_j}\to  0$.
From this and $y_j\to x$,
$$
 \abs{a_j-x}
 \leq \abs{a_j-y_j}+\abs{y_j-x}
 \leq R_{a_j}+ 1/j
 \to 0.
$$
Thus $a_j\to x$, so $x$ is an
accumulation point of $A$.  By Step~1, every accumulation point of
$A$ belongs to $K$.  This contradicts $x\notin K$ and
proves the observation.

Since $V_a\in C^{1,1}(B_1)$ is smooth away from $a$, it follows from the above observation that
$$
 W\in C^1(B_1\setminus K)\cap C^\infty(B_1\setminus E).
$$

Fix $a\in A$.  By the pairwise disjointness of the support balls and \eqref{core-and-cross-estimates-1},
$$
 u(x)=P(x)+\varepsilon_aw(x-a)
 \quad\text{for }\abs{x-a}<R_a/4.
$$
Since $w$ is not twice differentiable at the origin, $u$ is not twice differentiable
at $a$.

We now study $W$ at the points of $K$.
The situation here is subtler since the support balls accumulate to the points of $K$.  
Here is where we need to use \eqref{scaled-bump-estimates}. 

Let
$$ \textstyle
 G(y):=\sum_{a\in A}DV_a(y),\quad y\in B_1.
$$
This function 
is pointwise well-defined, since at most one summand is
nonzero at every point.  Clearly, we have
$DW=G$ in $B_1\setminus K$.

Fix $x\in K$. Since every
closed support ball is disjoint from $K$, we have
$W(x)=0$ and $G(x)=0$. Now we prove, for every $y\in B_1$,
\begin{equation}\label{K-point-bump-estimates}
 \abs{W(y)-W(x)}
 \leq C\abs{y-x}^3\quad \text{and}
 \quad
 \abs{G(y)-G(x)}
 \leq C\abs{y-x}^2.
\end{equation}
If $y$ lies outside all the closed support balls, then every $V_a(y)$
and every $DV_a(y)$ vanish. Therefore,
$W(y)=G(y)=0$, and thus \eqref{K-point-bump-estimates} clearly holds. So we may assume
$$
 y\in\overline{B(a,R_a)}
\quad 
\text{for some }a\in A,$$ 
then \eqref{radius-bounds} gives
$$
 \abs{y-x}
 \geq   \abs{a-x}-\abs{y-a}
 \geq \operatorname{dist}(a,K)-R_a
 \geq 9R_a.
$$
At the point $y$, we have
 $W(y)=V_a(y)$ and $G(y)=DV_a(y)$. 
From the above and \eqref{scaled-bump-estimates},
 \begin{align*}
 &\abs{W(y)-W(x)}
 =\abs{V_a(y)}
 \leq C\eta R_a^3
 \leq C\abs{y-x}^3,  \quad \text{and} \\
 &\abs{G(y)-G(x)}
 =\abs{DV_a(y)}
 \leq C\eta R_a^2
 \leq C\abs{y-x}^2.
 \end{align*}
Hence, \eqref{K-point-bump-estimates} is proved.

As a consequence of the first estimate in \eqref{K-point-bump-estimates}, we obtain the uniform estimate \eqref{cubic-expansion} for $x\in K$.
The first estimate in \eqref{K-point-bump-estimates} also implies that $W$ is differentiable at $x$ and
$DW(x)=G(x)=0$.
Since $x\in K$ was arbitrary and $DW=G$ was already known in
$B_1\setminus K$, it follows that $W$ is differentiable throughout
$B_1$ and
$$
 DW=G\quad\text{in }B_1.
$$
The second estimate in \eqref{K-point-bump-estimates} now gives that
$DW$ is differentiable at every $x\in K$, with derivative equal to
zero.  Hence $W$ is twice differentiable and
$D^2 W =0$ at every point of $K$. Therefore,
$$
 u \text{ is twice differentiable and }D^2u =D^2P \quad\text{for every point of } K.
$$

Now we prove that $DW=G$ is Lipschitz in $B_1$.  Let
$x,y\in B_1$.  
First suppose that $x$ and $y$ belong to the same
closed support ball $\overline{B(a,R_a)}$.  Then \eqref{scaled-bump-estimates} gives
 \begin{align*}
 \abs{G(x)-G(y)}
 &=\abs{DV_a(x)-DV_a(y)}\\
 &\leq C\varepsilon_a\abs{x-y}\leq C\abs{x-y}.
 \end{align*}

Suppose next that $x$ and $y$ do not belong to the same closed support
ball. If $x$ belongs to a closed support ball, denote its center by
$a$; if $x$ lies outside all the closed support balls, omit every term below
involving $a$.  Define $b$ analogously for $y$.  
By \eqref{scaled-bump-estimates}, we obtain
 \begin{align*}
 \abs{G(x)-G(y)}
 &\leq
 \abs{DV_a(x)-DV_a(y)}
 +
 \abs{DV_b(x)-DV_b(y)}\\
 &\leq
 C(\varepsilon_a+\varepsilon_b)\abs{x-y}\leq C\abs{x-y}.
 \end{align*}
Hence, $G=DW$ is Lipschitz in $B_1$.  Therefore, we have proved
$u\in C^{1,1}(B_1)$.

Finally, we prove $\Sing(u)=E$. 
Since $u$ is not twice differentiable at any point of $A$, we have $A\subset\Sing(u)$.
If $x\in K$, then every neighborhood of $x$ contains a point of $A$.  Since $u$ is not twice differentiable at that point, $u$
cannot be $C^2$ in any neighborhood of $x$.  Hence,
$K\subset\Sing(u)$, and thus $E\subset\Sing(u)$. 
Conversely, $u$ is smooth in $B_1\setminus E$, and therefore
$\Sing(u)\subset E$.  We conclude that
$\Sing(u)=E$. Claim 1 is proved. 
\end{proof}

By Step 1, the set $E=K\cup A$ is closed.  Moreover, since $K$ has
empty interior and $A$ is countable, $E$ has empty interior.  Thus
$E$ is nowhere dense in $B_1$.
Let
$$
 \Omega:=B_1\setminus E \neq \varnothing.
$$

\medskip
\noindent\textbf{Claim 2.}
There exists a constant $\theta>0$, independent of $K$, such that
$$
 \lambda_1\bigl(D^2u(x)-D^2u(y)\bigr)
 \geq\theta\norm{D^2u(x)-D^2u(y)}_F
 \quad\text{for all }x,y\in\Omega.
$$
In fact, one may choose $\theta=44^{-1}$ in this estimate. 

\begin{proof}[Proof of Claim 2]
Fix $x,y\in\Omega$.

{\it Case 1.}
There is $a\in A$ such that
$x,y\in B(a,R_a/4)$. In this case, write
$$
 x=a+re
 \quad\text{and}\quad
 y = a+ \rho f,
$$
where $r,\rho>0$ and $e,f\in\Sph^4$.  Using the linearity of $D^2P$, homogeneity of $w$, and \eqref{core-and-cross-estimates-1}, 
 \begin{align*}
 D^2u(x)-D^2u(y)
 &=rD^2P(e)-\rho D^2P(f)
   +\varepsilon_a\bigl(D^2w(e)-D^2w(f)\bigr)\\
 &=\varepsilon_a\left[
 K\left(e,\frac r{\varepsilon_a}\right)
 -K\left(f,\frac\rho{\varepsilon_a}\right)\right].
 \end{align*}
Applying Proposition
\ref{augmented-family-proposition}, we obtain 
$$
 \lambda_1\bigl(D^2u(x)-D^2u(y)\bigr)
 \geq {44}^{-1} \norm{D^2u(x)-D^2u(y)}_F.
$$

{\it Case 2.}
There is no $a\in A$ for which both $x$ and $y$ belong
to $B(a,R_a/4)$. 

For every $a\in A$, estimate
\eqref{core-and-cross-estimates-2} applies to $x$ and $y$.  At each of the
two points $x$ and $y$, at most one bump Hessian $D^2 V_a$ is nonzero.  Consequently, by the linearity of $D^2 P$ again,
$$
 D^2u(x)-D^2u(y)=D^2P(x-y)+\mathcal R,
$$
where the remainder matrix satisfies
\begin{equation}\label{remainder-bound}
\norm{\mathcal R}_F\leq2L\eta\abs{x-y} \le 2^{-1}{c_0}\abs{x-y}.
\end{equation}
In the above, the second inequality is from the choice of $\eta$.
Using these and \eqref{background-Hessian-estimate},
 $$
 \lambda_1\bigl(D^2u(x)-D^2u(y)\bigr)
 \geq\lambda_1\bigl(D^2P(x-y)\bigr)-\norm{\mathcal R}_F
 \geq  2^{-1}  {c_0} \abs{x-y}.
 $$
On the other hand, the identity
$\norm{D^2P(z)}_F=\sqrt{126}\abs{z}$ and
\eqref{remainder-bound} give
$$
 \norm{D^2u(x)-D^2u(y)}_F
 \leq\left(\sqrt{126}+ 2^{-1}{c_0} \right)\abs{x-y}.
$$
Combining the above two yields
$$
 \lambda_1\bigl(D^2u(x)-D^2 u(y)\bigr)
 \geq {c_0}({2\sqrt{126}+c_0})^{-1}
 \norm{D^2u(x)-D^2u(y)}_F.
$$

Now Claim 2 follows by taking
$\theta:=\min\left\{{44}^{-1},
 {c_0} (2\sqrt{126}+c_0)^{-1}\right\}=44^{-1}>0$.

\end{proof}

\medskip
\noindent\textbf{Step 3. Construction of the operator $F$.}

\begin{lemma}\label{operator-lemma}
For $n\geq2$, let $\mathcal H\subset\Sym(n)$ be nonempty and compact.
Suppose that there is $\theta>0$ such that
\begin{equation}\label{estimate-allow-F}
 \lambda_1(X-Y)\geq\theta\norm{X-Y}_F
\quad\text{for all }
X,Y\in\mathcal H.
\end{equation}
Then the operator
$F:\Sym(n)\to\R$ defined by
$$
 F(M)= \min_{X\in\mathcal H}
 \bigl\{\tr(M-X)+\sqrt{n}\theta^{-1}\lambda_1(M-X)\bigr\},\quad M\in \Sym(n), 
$$
satisfies $F=0$ on $\mathcal H$, and \eqref{eq:UE} with $\lambda=1$ and $\Lambda=1+\sqrt{n}\theta^{-1}$. 

\end{lemma}

\begin{proof}[Proof of Lemma \ref{operator-lemma}]
Let
$$
 \mu:= {\sqrt n}/{\theta},
 \quad\text{and}
 \quad G(M):=\tr M+\mu\lambda_1(M) \quad \text{for }
M\in\Sym(n).
$$
For $N\geq 0$, the Rayleigh quotient characterization of $\lambda_1$ gives
$0\leq\lambda_1(M+N)-\lambda_1(M)
 \leq\lambda_1(N)\leq\tr N$.
Consequently, 
$\tr N\leq G(M+N)-G(M)
 \leq(1+\mu)\tr N$. 

Define
$$
 F(M):=\min_{X\in\mathcal H}G(M-X)\quad\text{for } M\in \Sym(n).
$$
The minimum exists by the compactness of $\mathcal H$.  The ellipticity inequality of $G$ yields
$$
 \tr N\leq F(M+N)-F(M)
 \leq(1+\mu)\tr N
 \quad\text{for every }N\geq 0.
$$

It remains to show that $F$ vanishes on $\mathcal H$.  Fix
$X_0\in\mathcal H$.  
  It follows from the definition of $F$ that $F(X_0)\leq 0$. On the other hand, for every
$X\in\mathcal H$, by the Cauchy--Schwarz inequality and the hypothesis of the lemma,
\begin{align*}
 G(X_0-X)
 &=
 \tr(X_0-X)+\mu\lambda_1(X_0-X)\\
 &\geq
 -\sqrt n\norm{X_0-X}_F
 +\mu\theta\norm{X_0-X}_F
 =0.
\end{align*}
Thus $F(X_0)\geq 0$, and hence $F(X_0)=0$. The lemma is proved.
\end{proof}

Let
$$
 \mathcal H:=\overline{\{D^2u(x)\mid x\in\Omega\}}\subset\Sym(5).
$$ 
The set $\mathcal H$ is nonempty since $\Omega\neq \varnothing$.  It is also compact.  Indeed, $D^2P$ is bounded
in $B_1$, and at each $x\in\Omega$ at most one bump Hessian $D^2 V_a$ is nonzero.
For that bump Hessian, if it exists,
$$
 \norm{D^2V_a(x)}_F\leq C\varepsilon_a
 \leq {C\eta}/{10}.
$$
Therefore, $D^2 u(\Omega)$ is bounded, and therefore its closure
$\mathcal H$ is compact.

By continuity, passing to the limit in
Claim 2 gives \eqref{estimate-allow-F} for $\theta$ independent of $K$. 
Then applying Lemma \ref{operator-lemma} yields a uniformly
elliptic operator $F:\Sym(5)\to\R$ such that $F=0$ on $\mathcal H$.
In particular,
\begin{equation}\label{classical-equation-on-Omega}
 F(D^2 u)=0\quad\text{in } \Omega.
\end{equation}
Moreover, the ellipticity constants of $F$ depend
only on $\theta$, so are independent of $K$.

\medskip
\noindent\textbf{Step 4. Completion of the proof.}
We first prove
\begin{equation}\label{pointwise-equation-on-K}
 F(D^2u)=0\quad\text{on } K \quad \text{in the pointwise sense}.
\end{equation}
Fix $x\in K$.
By Step 1, there exists $a_j\in A$ such that $a_j\to x$.
Write $R_j:=R_{a_j}$. By \eqref{radius-bounds},
$$
 10 R_j\leq \operatorname{dist}(a_j,K)
 \le \abs{a_j-x},
$$
and thus $R_j\to 0$.  Fix any $e\in\Sph^4$, and let
$$
 y_j:=a_j+2R_je  \in B_1.
$$

We now show that all bump functions $V_a$ vanish in a neighborhood of $y_j$.
  The distance between
$\overline{B(a_j,R_j)}$ and $y_j$ is $R_j$.  If $a \in A$ and $a \neq a_j$, then
\eqref{radius-bounds-2} gives
 \begin{align*}
 \abs{y_j-a}-R_a
 &\geq\abs{a_j-a}-2R_j -R_a \\
 &\geq5(R_j+R_a)-2R_j -R_a =3R_j+4 R_a> 3R_j.
 \end{align*}
Moreover,
$$
 \operatorname{dist}(y_j,K)
 \geq\operatorname{dist}(a_j,K)-2R_j
 \geq8R_j.
$$
It follows that $B(y_j,R_j/2)\subset B_1$ is
disjoint from $K$, and is disjoint from every closed support ball. Therefore,
$$
y_j\in \Omega\quad \text{and}\quad  D^2u(y_j)=D^2P(y_j).
$$
Note also that $y_j\to x$, since
$\abs{y_j-x}\leq\abs{a_j-x}+2R_j$. 
Then the above implies $D^2 u(y_j)\to D^2 P(x)\in \mathcal H$.  Claim~1 gives
$D^2u(x)=D^2P(x)$, and therefore \eqref{pointwise-equation-on-K} is proved.

\medskip

To conclude the proof, we need the following removable singularity result, the proof of which borrows ideas from \cite{CLN3}. We also remark that the same proof leads to a much more general statement which will appear elsewhere. 

\begin{lemma}
\label{countable-viscosity-lemma}
For $n\geq 1$, let $A\subset B_1$ be countable,
and let $u\in C^1(B_1)$.  Suppose that $u$ is twice differentiable at every
point of $B_1\setminus A$ and
$$
 F(D^2 u)=0\quad \text{in } B_1\setminus A,
$$
where $F\in C^0$ is degenerate elliptic\footnote{We say $F:\Sym(n)\to \R$ is degenerate elliptic if $F(M+N)\ge F(M)$ for every $M\in \Sym(n)$ and every $N\ge 0$.}.
Then $u$ is a viscosity solution of \eqref{equ-u}. 
\end{lemma}

\begin{proof}[Proof of Lemma \ref{countable-viscosity-lemma}]
We only prove $F(D^2 u)\ge 0$ in $B_1$ in the viscosity sense, since the other inequality can be proved similarly. 
Let $\phi\in C^2$ touch $u$ from above at $x_0\in B_1$. Suppose the contrary that
$F(D^2\phi(x_0))<0$.  By continuity, there exist small $\delta,r>0$ such that
$$
u\le \phi  \quad \text{and}\quad 
 F(D^2\phi+2\delta I_n)<0\quad\text{on } \Omega\coloneqq \overline{B_r(x_0)}\subset B_1.
$$

Let $v(y):=u(y)-\phi(y)-\delta\abs{y-x_0}^2$.  Then $v(x_0)=0$ and
$v\leq-\delta r^2$ on $\partial\Omega$.  Take $0<\rho<\delta r/2$.  Since the set
$Dv(A\cap \Omega )$ is countable,
we may choose $p\in B_\rho(0)\setminus Dv(A\cap\Omega)$.

The function $w(y)\coloneqq  v(y)-p\cdot(y-x_0)$ attains its maximum on $\Omega$ at some $x\in\Omega$. Since $w(x_0)=0$ and
$w\le -\delta r^2+\rho r<0$ on $\partial\Omega$, the point $x$ must lie in the interior of $\Omega$. Therefore, $Dv(x)=p$, and the choice of $p$ gives $x\notin A$.
Since $w$ is twice differentiable at $x$, its local maximality at $x$ gives $D^2w(x)\le 0$, i.e., $D^2u(x)\leq D^2\phi(x)+2\delta I_n$. By the equation of $u$ and the degenerate ellipticity of $F$, we obtain
$$
 0=F(D^2u(x))\leq F(D^2\phi(x)+2\delta I_n)<0,
$$
a contradiction.  Therefore $F(D^2\phi(x_0))\ge 0$. The Lemma is proved.
\end{proof}

Finally, from Claim 1, \eqref{classical-equation-on-Omega}, and
\eqref{pointwise-equation-on-K}, an application of Lemma
\ref{countable-viscosity-lemma} gives \eqref{equ-u} 
in the viscosity sense. The proof of Theorem \ref{main-thm} is completed.

\section{Proof of Theorem~\ref{thm:main-2}}
\label{sec:proof-positive-result}

Let $\cP_2$ denote the space of polynomials on $\R^n$ of degree at most
two.
Following Armstrong--Silvestre--Smart \cite{ASC}, 
for an open set $U\subset\R^n$, $u\in C(U)$, and $x\in U$, define
$$
 \Psi(u,U)(x):=\inf\left\{A\geq0 \mid 
  \exists Q\in\cP_2 \text{ such that }
 \abs{u(y)-Q(y)}\leq 6^{-1}A\abs{x-y}^3,~\forall y\in U
 \right\},
$$
where the infimum of the empty set is $+\infty$.

Our proof of Theorem \ref{thm:main-2} uses the following $W^{3,\varepsilon}$ estimate of
Armstrong--Silvestre--Smart. 
A similar estimate was
used earlier by Caffarelli--Souganidis \cite{CaffarelliSouganidis}. 

\begin{proposition}\label{prop:third-order-tail}\textup{(\cite[Lemma~5.2]{ASC}).}
Let $F:\Sym(n)\to\R$ satisfy \eqref{eq:UE}, and let $u\in C(B_1)$ be a
viscosity solution of \eqref{equ-u} such that
$\norm{u}_{L^\infty(B_1)}\leq1$.  There are constants
$C>0$ and
$\varepsilon \in(0,1]$, depending only on $n$, $\lambda$ and $\Lambda$, such that
$$
 \left|\left\{x\in B_{1/2}\mid \Psi(u,B_1)(x)>t\right\}\right|
 \leq Ct^{-\varepsilon}
 \quad \text{for every }t>1.
$$
\end{proposition}

Although \cite[Lemma 5.2]{ASC} is formally stated under the standing
hypotheses of that paper, which include \eqref{F-C1}, its proof uses
only uniform ellipticity.  Indeed, the interior $C^{1,\gamma}$ estimate and stability of viscosity solutions show, by passing to the limit in difference
quotients, that every directional derivative of $u$ satisfies a
Pucci extremal inequality; see
\cite[Proposition 2.9, Proposition 5.5, and Corollary 5.7]
{Caffarelli1995FullyNE}.  Applying the $W^{2,\epsi}$ estimate 
\cite[Proposition 3.1]{ASC} to the directional derivatives, together with an elementary relation \cite[Lemma 5.1]{ASC}, yields the proposition.

We record the following standard estimate for quadratic polynomials. 

\begin{lemma}\label{lem:quadratic-polynomial-estimate}
There exists a constant $C_n>0$ such that, for every $Q\in\cP_2$, $x\in\R^n$, and $r>0$,
$$
 \abs{Q(x)}+r\abs{DQ(x)}+r^2\norm{D^2Q(x)}_F
 \leq C_n\norm{Q}_{L^\infty(B_r(x))}.
$$
\end{lemma}
The proof of this lemma is elementary by using the expressions of $Q(x\pm r e)$ for $\abs{e}=1$. We omit the details.

In the proof of Theorem \ref{thm:main-2},
we use the following scaled error of best quadratic approximation.
For $x\in\R^n$, $r>0$, and $u\in C(\overline{B_r(x)})$, define
$$
 \omega_u(x,r):=\frac1{r^2}
 \inf_{Q \in\cP_2}\norm{u-Q}_{L^\infty(B_r(x))}.
$$
The infimum is attained.  Indeed, a minimizing sequence is uniformly bounded
on $B_r(x)$, and Lemma~\ref{lem:quadratic-polynomial-estimate} gives a uniform
bound for all of its coefficients. Passing to a subsequence then yields a minimizer.

The following standard telescoping argument gives a criterion
for the existence of a $C^{2,\alpha}$ expansion; compare
\cite[proof of Proposition~4.1]{ASC}.
 
\begin{lemma}\label{lem:decreasing-scales}
Let $x\in\R^n$, $\alpha\in [0,1)$, and let $r_k=2^{-k}r_0>0$ for $k\ge 0$. 
If $u\in C(\overline{B_{r_0}(x)})$ satisfies
$$
 \omega_u(x,r_k)\leq k^{-2}r_k^\alpha
 \quad\text{for all sufficiently large }k,
$$
then $u$ has a $C^{2,\alpha}$ expansion at $x$.
\end{lemma}

\begin{proof}
After replacing $r_0$ by one of the $r_k$ and relabeling, we may assume
that the hypothesis holds for every $k\geq1$.  Let $P_k\in \cP_2$ be a minimizer in the definition of $\omega_u(x,r_k)$, denoted by
$$
 P_k(x+h)=c_k+p_k\cdot h+ 2^{-1} h^TA_kh.
$$
Then, by the lemma hypothesis and the triangle inequality,
$$
 \norm{P_{k+1}-P_k}_{L^\infty(B_{r_{k+1}}(x))}
 \leq 2 k^{-2}r_k^{2+\alpha}.
$$
Applying Lemma \ref{lem:quadratic-polynomial-estimate} at radius
$r_{k+1}=r_k/2$ gives
\begin{align*}
 \abs{c_{k+1}-c_k} \leq Ck^{-2}r_k^{2+\alpha},\quad
 \abs{p_{k+1}-p_k} \leq Ck^{-2}r_k^{1+\alpha},\quad \text{and}\quad 
 \norm{A_{k+1}-A_k}_F \leq Ck^{-2}r_k^\alpha.
\end{align*}
Here and below, the constant $C$ depends only on $n$, and may vary from line to line. 
From the above, all the three coefficient sequences converge. Denote their limits by
$c,p,A$, and let
$$P_x(x+h)=c+p\cdot h+2^{-1}h^TA h. $$

Summing the coefficient estimates gives
$$
 \abs{c-c_k}\leq Cr_k^2 S_k,
 \quad
 \abs{p-p_k}\leq Cr_k S_k,
 \quad \text{and} \quad
 \norm{A-A_k}_F\leq C S_k, 
$$
where $S_k :=\sum_{j=k}^\infty j^{-2}r_j^\alpha=o(r_k^\alpha)$. The factor $k^{-2}$ in the hypothesis yields this little-$o$ estimate.
From the above and the lemma hypothesis again, we obtain
$$
 \norm{u-P_x}_{L^\infty(B_{r_k}(x))}
 =o(r_k^{2+\alpha}).
$$
Moreover,
$\abs{c_k-u(x)}\leq r_k^2\omega_u(x,r_k)\to0$, so $c=u(x)$.  If
$r_{k+1}<\abs{h}\leq r_k$, then $r_k<2\abs{h}$, and the above
estimate yields
$u(x+h)-P_x(x+h)=o(\abs{h}^{2+\alpha})$.  This is the desired expansion.
\end{proof}

The next lemma shows that a large quadratic approximation error on one ball
forces $\Psi$ to be large on its concentric half ball.

\begin{lemma}\label{lem:bad-scale}
Let $x\in\R^n$ and $r>0$. If $u\in C(\overline{B_r(x)})$ satisfies 
$$\omega_u(x,r)>a>0,$$
then
$\Psi(u,B_r(x)) > a/r$ on $B_{r/2}(x)$.
\end{lemma}

\begin{proof}
Fix $y\in B_{r/2}(x)$. If $\Psi(u,B_r(x))(y)=+\infty$, there is nothing to prove.
We may therefore assume that $\Psi(u,B_r(x))(y)<+\infty$.
Then, for every $\delta>0$, there exists
$Q\in\cP_2$ such that
$$
 \abs{u(z)-Q(z)}
 \leq 6^{-1}\big(\Psi(u,B_r(x))(y)+\delta\big)\abs{z-y}^3
 \quad\text{for every }z\in B_r(x).
$$
Since $\abs{z-y}\leq3r/2$ for $z\in B_r(x)$, it follows that
$$
 \omega_u(x,r)
 \leq \frac{1}{r^2}
       \norm{u-Q}_{L^\infty(B_r(x))}
 \leq \frac9{16}r\big(\Psi(u,B_r(x))(y)+\delta\big).
$$
Sending $\delta$ to zero yields
$$
 \Psi(u,B_r(x))(y)
 \geq\frac{16}{9r}\omega_u(x,r)
 >\frac{a}{r}.
$$
The lemma is therefore proved.
\end{proof}

Now we are ready to give the proof of Theorem \ref{thm:main-2}. Our terminology and notation related to Hausdorff measure follow those in \cite[Chapter 2]{EvansGariepy}.

\begin{proof}[Proof of Theorem~\ref{thm:main-2}]
Assume first that $\norm{u}_{L^\infty(B_1)}\leq1$. We remove this normalization at the end of this proof.  Fix
$0\leq\alpha<1$ and $0<r_0<1/4$. For every $k\ge 1$, let
$$
 r_k:=2^{-k}r_0,
 \quad
 a_k:=k^{-2}r_k^\alpha,
 \quad \text{and} \quad 
 E_k:=\{x\in B_{1/4}\mid \omega_u(x,r_k)>a_k\}.
$$
Then Lemma \ref{lem:decreasing-scales} yields
\begin{equation}\label{eq:limsup-bad-sets} 
 \Sigma_{2,\alpha}(u)\cap B_{1/4}
 \subset\bigcap_{K=1}^\infty\bigcup_{k=K}^\infty E_k.
\end{equation}

For each $k$, we prove that there exist finitely many points $x_{k,1}, \dots, x_{k,N_k}$ in $E_k$ such that the balls $B_{r_k/2}(x_{k,i})\subset B_{1/2}$ are pairwise disjoint, and 
$$ \textstyle
 E_k\subset\bigcup_{i=1}^{N_k}B_{r_k}(x_{k,i}).
$$
Indeed, if $E_k$ is empty, we simply take $N_k=0$. So we may assume that $E_k$ is nonempty.
Choose $x_{k,1}\in E_k$.
After $x_{k,1},\ldots,x_{k,j}$ have been chosen, stop if
$E_k\subset\bigcup_{i=1}^jB_{r_k}(x_{k,i})$; if this inclusion fails, choose
an $x_{k,j+1}$ in the difference.  Distinct chosen points are at distance at least $r_k$, so the balls $B_{r_k/2}(x_{k,i})$ are pairwise disjoint.  They
are contained in $B_{1/2}$ because $x_{k,i}\in B_{1/4}$ and
$r_k/2<1/4$.  Hence, after any $j$ choices, it holds that
$j\abs{B_{r_k/2}}\leq\abs{B_{1/2}}$.  Thus the procedure must stop after
finitely many choices, with $E_k$ covered by $\bigcup_{i=1}^{N_k}B_{r_k}(x_{k,i})$.

For each $i$, by Lemma \ref{lem:bad-scale} and ${B_{r_k/2}(x_{k,i})}\subset B_{1/2}$, we obtain
$$
 B_{r_k/2}(x_{k,i})
 \subset\left\{y\in B_{1/2} \mid 
 \Psi(u,B_1)(y)> {a_k}/{r_k}\right\}.
$$
From this and the pairwise disjointness of the
smaller balls $B_{r_k/2}(x_{k,i})$, applying Proposition \ref{prop:third-order-tail} yields, for
all sufficiently large $k$,
\begin{align*}
 N_k\abs{B_{r_k/2}}
 &\leq\left|\left\{y\in B_{1/2} \mid
 \Psi(u,B_1)(y)> {a_k}/{r_k}\right\}\right|\\
 &\leq C\left({a_k}/{r_k}\right)^{-\varepsilon}.
\end{align*}
Here we used $\alpha<1$ for $a_k/r_k=k^{-2}r_k^{\alpha-1}\to\infty$, and positive constants $\epsi, C$ depend only on $n$, $\lambda$, and $\Lambda$. Here and below, $C$ may vary from line to line, while
$\varepsilon$ denotes the fixed exponent in
Proposition~\ref{prop:third-order-tail}.
It follows from the above that 
$$
 N_k\leq C k^{2\varepsilon}
 r_k^{-n+\varepsilon(1-\alpha)}\quad \text{for all sufficiently large }k.
$$

Let $s>n-\varepsilon(1-\alpha)$. For every sufficiently large $K$,  
\eqref{eq:limsup-bad-sets} implies that the balls
$\{B_{r_k}(x_{k,i})\}_{k\ge K,~ 1\le i\le N_k}$ cover
$\Sigma_{2,\alpha}(u)\cap B_{1/4}$.  Thus the definition of Hausdorff
measure gives
\begin{align*}
 \cH^s_{2r_K}(\Sigma_{2,\alpha}(u)\cap B_{1/4})
 \leq C\sum_{k=K}^\infty N_kr_k^s
 \leq C\sum_{k=K}^\infty
 k^{2\varepsilon}r_k^{s-n+\varepsilon(1-\alpha)}.
\end{align*}
The last exponent is positive, so the right hand side
tends to zero as $K\to\infty$.   Hence,
\begin{equation}\label{eq:local-Hausdorff-zero}
 \cH^s(\Sigma_{2,\alpha}(u)\cap B_{1/4})=0
 \quad\text{for every }s>n-\varepsilon(1-\alpha).
\end{equation}
This proves that the Hausdorff dimension of $\Sigma_{2,\alpha}(u)\cap B_{1/4}$ is at most $n-\epsi(1-\alpha)$. 

Finally, we remove the normalization on $\norm{u}_{L^\infty}$.  Let $u$ be a general solution in the statement of Theorem \ref{thm:main-2}. We claim that for every ball $B_R(x_0)\Subset B_1$, 
\begin{equation}\label{Haus-estimate-local}
 \cH^s(\Sigma_{2,\alpha}(u)\cap B_{R/4}(x_0))=0
 \quad\text{for every }s>n-\varepsilon(1-\alpha).
\end{equation}

Assuming this claim, we now conclude the proof.
The ball $B_1$ has a countable covering by those balls $B_{R/4}(x_0)$ in \eqref{Haus-estimate-local}.
Countable subadditivity yields
$\cH^s(\Sigma_{2,\alpha}(u))=0$ for every $s>n-\varepsilon(1-\alpha)$, and therefore
 $\dim_{\cH}\Sigma_{2,\alpha}(u)
 \leq n-\varepsilon(1-\alpha)$.

It remains to prove the claim.
Indeed, let
$L:=\norm{u-u(x_0)}_{L^\infty(B_R(x_0))}$. We may assume $L>0$, since otherwise $u$ is constant in the ball, and \eqref{Haus-estimate-local} follows.
Consider
$$
 \widehat u(z): = L^{-1}\bigl( {u(x_0+Rz)-u(x_0)}\bigr)
 \quad\text{and}\quad 
 \widehat F(M):=L^{-1}{R^2}
 F\left(L{R^{-2}}M\right).
$$
Then $\widehat F$ has the same ellipticity constants as $F$,
$\norm{\widehat u}_{L^\infty(B_1)}\leq1$, and
$\widehat F(D^2\widehat u)=0$ in $B_1$.  
Therefore,
\eqref{eq:local-Hausdorff-zero} gives $\cH^s(\Sigma_{2,\alpha}(\widehat{u})\cap B_{1/4})=0$ for every $s> n-\epsi (1-\alpha)$. 
Note that $\widehat{u}$ has a $C^{2,\alpha}$ expansion at $z\in B_1$ if and only if $u$ has a $C^{2,\alpha}$ expansion at $x_0+Rz\in B_R(x_0)$. Then the translation invariance and the scaling property of 
Hausdorff measure yield \eqref{Haus-estimate-local}.

\end{proof}

\appendix

\section{Proof of Lemma~\ref{orbit-reduction}}
\label{appendix-orbit-reduction}

We provide a proof somewhat different from the one in \cite{NTV}.
Let $\Sym_0(3)$ be the space of traceless symmetric $3\times 3$ matrices, endowed
with the Frobenius inner product, and define
$$
 J(x):=
 \begin{pmatrix}
 \dfrac{x_1}{\sqrt6}+\dfrac{x_2}{\sqrt2}
 &-\dfrac{z_3}{\sqrt2}&\dfrac{z_2}{\sqrt2}\\[2mm]
 -\dfrac{z_3}{\sqrt2}
 &\dfrac{x_1}{\sqrt6}-\dfrac{x_2}{\sqrt2}
 &\dfrac{z_1}{\sqrt2}\\[2mm]
 \dfrac{z_2}{\sqrt2}&\dfrac{z_1}{\sqrt2}
 &-\dfrac{2x_1}{\sqrt6}
 \end{pmatrix}
$$
for $x=(x_1,x_2,z_1,z_2,z_3)\in\R^5$. 
One may verify $$ \norm{J(x)}_F=\abs{x}.$$ Thus
$J:\R^5\to\Sym_0(3)$ is a linear isometry.  A
direct computation gives
\begin{equation}\label{appendix-J-identities}
 P(x)=-3\sqrt6\det J(x).
\end{equation}

Fix $e\in\Sph^4$, and let $a,b,c$ be the eigenvalues of $J(e)$.  Then
$a+b+c=0$ and $a^2+b^2+c^2=1$.
Using $a+b+c=0$, the elementary identity
$$
 (a-b)^2(b-c)^2(c-a)^2
 =2^{-1}(a^2+b^2+c^2)^3-27a^2b^2c^2
$$
shows that $\abs{abc}\leq1/(3\sqrt6)$.  Therefore,
\eqref{appendix-J-identities} gives $P(e)\in[-1,1]$.

The function
$$
 g(p):= p(3-p^2)/2
$$
is strictly increasing from $[-1,1]$ onto itself.  Hence, there is a
unique $p\in[-1,1]$ such that $P(e)=g(p)$.  Direct substitution in
\eqref{cartan-cubic} gives $P(e_p)=g(p)$.
Hence,
the matrices $J(e)$ and $J(e_p)$ have the same Frobenius norm and
determinant.  Since the characteristic
polynomial of a traceless $3\times 3$ matrix $Y$ is of the form 
$\det(\lambda I_3-Y)
 =\lambda^3- 2^{-1}\norm{Y}_F^2\lambda-\det Y$,
matrices $J(e)$ and $J(e_p)$ have the same eigenvalues. By the spectral theorem, there is $Q\in O(3)$ such that
$$J(e)=QJ(e_p)Q^T.$$
Define $O:\R^5\to\R^5$ by
$J(Ox)=QJ(x)Q^T$.
Since $J$ is an isometry, we obtain $O\in O(5)$ and $Oe_p=e$. By
\eqref{appendix-J-identities}, we have
$P(Ox)=-3\sqrt6\det\bigl(QJ(x)Q^T\bigr)=P(x)$. 
The lemma is proved.

\section{\texorpdfstring{On the notion of pointwise
$C^{2,\alpha}$ expansion}{On the notion of pointwise C2-alpha expansion}}
\label{app-expansion}

\begin{lemma}
\label{prop:punctual-versus-twice}
If $u\in C^{1,1}_{\mathrm{loc}}(B_1)$ has a $C^{2,0}$ expansion at
$x_0\in B_1$, then $u$ is twice differentiable at $x_0$.
\end{lemma}

The $C^{1,1}_{\mathrm{loc}}$ assumption in the above lemma cannot be weakened to $C^{1,\alpha}_{\mathrm{loc}}$ for any $\alpha<1$, since otherwise the existence of a $C^{2,0}$ expansion need not imply twice differentiability, as shown by the following example. In this paper, $C^{1,0}_{\mathrm{loc}}$ and $C^{1,0}$ should be read as $C^1$.

\begin{example}\label{ex-b2}
     For every $n\ge 1$ and every $0\leq\alpha<1$, there exists $u\in C^{1,\alpha}(B_1)$ satisfying 
$u(h)=o(\abs{h}^2)$ as $h\to 0$,
but $u$ is not twice differentiable at $0$.

Indeed, it suffices to construct such an example when $n=1$, since the function $x\mapsto u(x_1)$ then gives an example in every higher dimension.   
Let $b:=(1-\alpha)/(1+\alpha)>0$ and $a:=b+ 2^{-1}(3+\alpha)= 2+ 2^{-1}b(1-\alpha)>2$. Define
$$
 u(t):=
  t^a\sin(t^{-b})~\text{ when}~t>0,\quad \text{and}\quad 
  u(t):=0~\text{when}~t\leq  0.
$$
Then $u\in C^{1,\alpha}(-1,1)$ satisfies $u(t)=o(t^2)$, but $u$ is not twice differentiable at $0$. 
\end{example}

\begin{proof}[Proof of Lemma \ref{prop:punctual-versus-twice}]
For $|h|$ small, let
$$
 v(h):=u(x_0+h)-u(x_0)- D u(x_0)\cdot h- 2^{-1} h^TAh,
$$
where $A$ is given by the $C^{2,0}$ expansion of $u$ at $x_0$. 
It follows that $v(h)=o(\abs{h}^2)$ as $h\to 0$.
Since $u\in C^{1,1}_{\mathrm{loc}}$, we have $v\in C^{1,1}$ in
a neighborhood of $0$.  It remains
to prove $Dv(h)=o(\abs{h})$.

Fix a $\delta>0$ small. We may assume $$ \omega(\epsi)\coloneqq \sup_{B_\epsi\setminus\{0\}} \frac{|v(h)|}{|h|^2}>0\quad \text{for every }0<\epsi<2\delta, $$
since otherwise the desired conclusion is clear. By $v(h)=o(|h|^2)$,  we have $\omega(\epsi)\downarrow 0$ as $\epsi\downarrow 0$. 
After shrinking $\delta$ if necessary, we may assume $\omega(2\delta)<1$.
For $h\in B_\delta\setminus\{0\}$, let $t\coloneqq |h|\sqrt{\omega(2|h|)}$. Let $L$ be a Lipschitz constant of $Dv$ in $B_{2\delta}$. Then, for each unit vector $e$, we obtain
$$t\abs{D v(h)\cdot e} \le \abs{v(h+t e)}+ \abs{v(h)}+L t^2\le (5+L) t^2,\quad h\in B_\delta\setminus\{0\}. $$
Hence, $\abs{D v(h)}\le (5+L) |h| \sqrt{\omega(2|h|)}$ for all $0<|h|<\delta$. Thus $Dv(h)=o(|h|)$. 
\end{proof}

\begin{lemma}\label{lem:punctual-implies-second}
Let $F:\Sym(n)\to\R$ satisfy \eqref{eq:UE}, and let $x_0\in B_1$.  
If $u$ is a continuous
viscosity solution of \eqref{equ-u}, then $u$ is twice differentiable at $x_0$ if and only if $u$ has a $C^{2,0}$ expansion at $x_0$.
\end{lemma}

\begin{proof}
The interior $C^{1,\alpha}$ regularity in 
\cite[Corollary~5.7]{Caffarelli1995FullyNE} implies $u\in C^{1}(B_1)$. Hence, the implication from twice
differentiability to the existence of a $C^{2,0}$ expansion is clear. Conversely, suppose that
$u(x_0+h)= Q(h) +o(\abs{h}^2)$ as $h\to 0$, where $Q(h)\coloneqq u(x_0)+Du(x_0)\cdot h+2^{-1} h^TAh$. 
  We prove that $u$ is
twice differentiable at $x_0$.
For every
$\epsi>0$, $Q+\epsi\abs{h}^2$ and
$Q-\epsi \abs{h}^2$ touch $u(x_0 + \cdot)$ at $0$ from above and
below, respectively. Since $u$ is a viscosity solution of \eqref{equ-u}, we obtain 
$F(A+2\epsi I_n)\geq 0$ and $F(A-2\epsi I_n)\leq 0$.  Sending $\epsi$ to zero, we obtain $F(A)=0$.
For $0<r<\operatorname{dist}(x_0,\partial B_1)$, let
$$
 \omega(r):=\sup_{h\in B_1\setminus\{0\}}\frac{\abs{v_r(h)}}{\abs{h}^2},
 \quad\text{and}\quad 
 v_r(h):=\frac{{u(x_0 +r h)-Q(rh)}}{r^2},~h~\in B_1.
$$
Then $\omega(r)\to0$ and
$\norm{v_r}_{L^\infty(B_1)}\leq\omega(r)$.  Moreover, $v_r$ solves
$F_A(D^2v_r)=0$ in $B_1$, where $F_A(M):=F(A+M)$.  This operator has
the same ellipticity constants as $F$ and satisfies $F_A(0)=0$.
Therefore, the interior $C^1$ estimate gives
$$
 \|Dv_r\|_{L^\infty(B_{1/2})}
 \leq C\|v_r\|_{L^\infty(B_1)}
 \leq C\omega(r)
 =o(1)
 \quad\text{as }r\to 0^+.
$$
For $h\neq 0$ sufficiently small, take $r=3\abs{h}$ and
$y=h/r\in B_{1/2}$.  Since
$Dv_r(y)=r^{-1}\bigl(Du(x_0+h)-Du(x_0)-Ah\bigr)$, it follows that
$Du(x_0+h)=Du(x_0)+Ah+o(\abs{h})$.  Thus $u$ is twice
differentiable at $x_0$.
\end{proof}

Recall that a subset $E$ of $\R^n$ is called a
\emph{$G_{\delta\sigma}$ set} if it is a countable union of $G_\delta$ sets, where a \emph{$G_\delta$ set} is a countable
intersection of open sets in $\R^n$.

\begin{lemma}\label{Lem-G-set}
If $u\in C^0(B_1)$, then $\Sigma_{2,\alpha}(u)$ is a
$G_{\delta\sigma}$ set for every $0\leq\alpha<1$.
\end{lemma}

\begin{proof}
Let $\mathcal G_{2,\alpha}(u):=B_1\setminus\Sigma_{2,\alpha}(u)$.  For
$m\geq1$ and $k\geq3$, let $E_{m,k}$ be the set of points
$x\in\overline B_{1-2/k}$ for which there are $p\in\R^n$ and
$A\in\Sym(n)$ satisfying
$$
 \left|u(x+h)-u(x)-p\cdot h-2^{-1} h^TAh \right|
 \leq  m^{-1} \abs{h}^{2+\alpha}
 \quad\text{whenever }\abs{h}\leq k^{-1}.
$$
Each $E_{m,k}$ is closed.  Indeed, suppose
$x_j\in E_{m,k}$ and $x_j\to x$, and choose corresponding
$p_j,A_j$.  Let $r=1/(2k)$. Evaluating at $h=\pm re$ shows, uniformly
for unit vectors $e$, that
$$
 2r\abs{p_j\cdot e}
 \leq\abs{u(x_j+re)-u(x_j-re)}+ 2 m^{-1} r^{2+\alpha},\quad
\text{and}
$$
$$
 r^2\abs{e^TA_je}
 \leq\abs{u(x_j+re)+u(x_j-re)-2u(x_j)}
      + 2m^{-1} r^{2+\alpha}.
$$
Continuity of $u$ on a compact subset of $B_1$ therefore
bounds $p_j$ and $A_j$.  Passing to a convergent subsequence and then
to the limit in the defining inequality proves that $x\in E_{m,k}$.

It remains to prove the following claim:
$$
 \mathcal G_{2,\alpha}(u)
 =\cap_{m\ge 1}\cup_{k\ge 3} E_{m,k}.
$$
The forward inclusion follows directly from the definition of a
$C^{2,\alpha}$ expansion.  Conversely, suppose that $x$ belongs to
the right hand side.  For every $m$, choose $k_m,p_m,A_m$ as in the
definition of $x\in E_{m,k_m}$. 
Let $m,\ell\geq1$, $0<t<\min\{1/k_m,1/k_\ell\}$, and let $e$ be a unit vector. Then comparison at $h=te$ and $h=-te$
gives
$$
 \abs{(p_m-p_\ell)\cdot e}
 \leq\left(m^{-1}+\ell^{-1}\right)t^{1+\alpha}
 \quad\text{and}\quad 
 \abs{e^T(A_m-A_\ell)e}
 \leq2\left( m^{-1}+ \ell^{-1}\right)t^\alpha.
$$
Letting $t\downarrow0$ in the first inequality shows that all $p_m$
are equal to a single vector $p$.  If $\alpha>0$, the second inequality
similarly shows that all $A_m$ are equal to a single matrix $A$.  If
$\alpha=0$, it instead gives
$\norm{A_m-A_\ell}_F\leq2\sqrt n(1/m+1/\ell)$, so $A_m\to A$ and
$\norm{A_m-A}_F\leq2\sqrt n/m$.  In either case, the defining
inequality for $x\in E_{m,k_m}$ implies, for $\abs{h}\leq1/k_m$,
$$
 \left|u(x+h)-u(x)-p\cdot h- 2^{-1} h^TAh\right|
 \le
 (1+\sqrt n)m^{-1}\abs{h}^{2+\alpha}.
$$
Since $m$ is arbitrary, $u$ has a $C^{2,\alpha}$ expansion at $x$,
which proves the claim.
\end{proof}

\section*{Acknowledgment}
I would like to thank my thesis adviser, Professor YanYan Li, for introducing Question \ref{Q-1} to me and encouraging me to think about it when I was a second-year PhD student at Rutgers University.

\end{document}